\documentclass[reqno,11pt]{amsart}

\usepackage{tikz}
\usepackage{tikz-cd}
\usepackage{amssymb}
\usepackage{amsgen}
\usepackage{amsmath}
\usepackage{amsthm}
\usepackage{mathrsfs}
\usepackage{cite}
\usepackage{amsfonts}

\newcommand{\sour}{\mathop{\boldsymbol s}}
\newcommand{\Bis}{\Gamma}
\newcommand{\skel}[1]{^{(#1)}}

\newcommand{\dom}{\mathop{\boldsymbol d}}
\newcommand{\ran}{\mathop{\boldsymbol r}}

\renewcommand{\to}{\longrightarrow}

\newcommand{\supp}{\mathop{\mathrm{supp}}}

\newcommand{\inv}{^{-1}}
\newcommand{\p}{\varphi}
\newcommand{\pinv}{{\p \inv}}

\newcommand{\wh}{\widehat}

\usepackage{xcolor}

\newtheorem{Thm}{Theorem}[section]
\newtheorem{Prop}[Thm]{Proposition}

{\theoremstyle{definition}
}
{\theoremstyle{remark}
\newtheorem{Rmk}[Thm]{Remark}}
\newtheorem{Cor}[Thm]{Corollary}
{\theoremstyle{remark}
}
{\theoremstyle{remark}
}

{\theoremstyle{remark}
}
{\theoremstyle{remark}
}
{\theoremstyle{remark}
}

\numberwithin{equation}{section}

\title{Diagonal-preserving isomorphisms of \'etale groupoid algebras}

\author{Benjamin Steinberg}
\address[B.~Steinberg]{%
    Department of Mathematics\\
    City College of New York\\
    Convent Avenue at 138th Street\\
    New York, New York 10031\\
    USA}
\email{bsteinberg@ccny.cuny.edu}

\thanks{The author was supported by  NSA MSP \#H98230-16-1-0047 and by a PSC CUNY grant.}
\date{November 6, 2017}


\begin{document}

\begin{abstract}
Work of Jean Renault shows that, for topologically principal \'etale groupoids, a diagonal-preserving isomorphism of reduced $C^*$-algebras yields an isomorphism of groupoids.  Several authors have proved analogues of this result for ample groupoid algebras over integral domains  under suitable hypotheses.  In this paper, we extend the known results by allowing more general coefficient rings and by weakening the hypotheses on the groupoids.  Our approach has the additional feature that we only need to impose conditions on one of the two groupoids.  Applications are given to Leavitt path algebras.
\end{abstract}

\maketitle

\section{Introduction}
Groupoid $C^*$-algebras have played a prominent role in the theory of operator algebras since Renault's seminal monograph~\cite{Renault}; see also Connes~\cite{Connes}.   In recent years, there has been a flurry of activity~\cite{operatorsimple1,operatorguys2,reconstruct,GroupoidMorita,CarlsenSteinberg,Strongeffective,ClarkPardoSteinberg,groupoidbundles,groupoidprimitive,Clarkdecomp,GonRoy2017a,GonRoy2017,Hazrat2017,AraSims2017,Purelysimple,gpdchain} around groupoid algebras over commutative rings~\cite{mygroupoidalgebra,operatorguys1}, as this class encompasses group algebras, Leavitt path algebras~\cite{LeavittBook,Abramssurvey,Leavitt1,LeavittPardo} and inverse semigroup algebras amongst other interesting rings.

Groupoid algebras are built from an ample groupoid $\mathscr G$ and a commutative ring with unit $R$.  An ample groupoid is an \'etale groupoid with a totally disconnected, locally compact unit space.  If $\mathscr G$ and $\mathscr G'$ are isomorphic ample groupoids, then there is a diagonal-preserving isomorphism $R\mathscr G\to R\mathscr G'$ of algebras, where the diagonal subalgebra of a groupoid algebra is the commutative algebra of compactly supported locally constant functions on the unit space.  Groupoid reconstruction is concerned with determining to what extend the converse holds.

Renault~\cite{renaultcartan} proved in the operator theoretic setting that a diagonal-preserving isomorphism between reduced $C^*$-algebras implies an isomorphism of groupoids provided the groupoids are topologically principal.  This result has inspired most of the work in the algebraic setting.

Prior to this paper, all of the work has focused on the case where $R$  is an integral domain.  Of course, the isomorphism problem for group rings (when does $\mathbb ZG_1\cong \mathbb ZG_2$ imply that the groups $G_1,G_2$ are isomorphic) is exactly our problem in the special case that the ample groupoids are discrete groups and the ring $R$ is the integers.  There is a famous example~\cite{isointegralgroup} of two finite non-isomorphic groups with isomorphic integral group rings and so diagonal-preserving isomorphisms cannot recover the groupoid in general; simpler examples are, of course, available over the field of complex numbers.  Note that if Kaplansky's Unit Conjecture~\cite{PassmanBook} is true, then a torsion-free group is determined by its algebra and diagonal subalgebra over any integral domain.

In~\cite{BCvH2017} it was shown that a diagonal-preserving isomorphism of $\ast$-rings between Leavitt path algebras implies an isomorphism of the corresponding path groupoids. Of course, a groupoid isomorphism yields a diagonal-preserving isomorphism of $\ast$-rings.  However, we shall not require the $\ast$-structure to be preserved in this paper.  The paper~\cite{reconstruct} shows that, for topologically principal groupoids, a diagonal-preserving ring isomorphism between groupoid algebras gives rise to an isomorphism of groupoids; this is the groupoid analogue of Renault's theorem~\cite{renaultcartan}.  Topologically principal groupoids are a special case of the important class of effective groupoids.  For second countable groupoids, which is what operator theorists usually consider, being effective is equivalent to being topologically principal, but in general it is a weaker condition~\cite{operatorsimple1}.

The strongest result prior to our work is that of Carlsen and Rout~\cite{CarlsenSteinberg}.  They prove that groupoid reconstruction is possible as long as there is a dense set of objects such that the group ring of the isotropy group at each of these objects has no zero divisors and no non-trivial units.   Diagonal-preserving isomorphisms are classified in~\cite{cordeiro} for this context.  We remark that over an integral domain, having no zero divisors and no non-trivial units is equivalent to being torsion-free and having no non-trivial units~\cite{PassmanBook}.  This condition is satisfied by any left or right orderable group and, in particular, by any torsion-free abelian group.  The results of Carlsen and Rout imply those of~\cite{BCvH2017,reconstruct} discussed above, but they are not strong enough to cover all effective groupoids since, in principle, a groupoid can be effective and fail the no zero divisors and no non-trivial units condition at each object.
The results on diagonal-preserving isomorphisms of Leavitt path algebras were put to good effect in the paper~\cite{CRS17} in connection with the work of Matsumoto and Matui on symbolic dynamics and groupoids, cf.~\cite{MatsMatu}.

We generalize the results of Carlsen and Rout in several ways.  Firstly, we weaken the hypotheses on the base ring $R$.  Secondly,  we show that one can work with the group rings of the isotropy groups of the interior of the isotropy bundle; this will allow us to apply our results to all effective groupoids.  Also, we do not require the group rings to have no zero divisors; we just need no non-trivial units.  For example, the integral group rings of finite abelian groups of exponent $4$ or $6$, as well as direct products of quaternion groups or left orderable groups with elementary abelian $2$-groups~\cite{higmanunit}, have no non-trivial units  but have zero divisors and so give examples to which our results apply but those of~\cite{CarlsenSteinberg} do not. Finally, our results only require one of the two groupoids linked by a diagonal-preserving isomorphism of algebras to satisfy our hypotheses in order to conclude that the groupoids are isomorphic; previous papers~\cite{BCvH2017,reconstruct,CarlsenSteinberg} required both groupoids to satisfy their hypotheses.

Despite working in a more general setting, we believe that our proofs are technically simpler than those in~\cite{BCvH2017,reconstruct,CarlsenSteinberg}.  Our chief innovation is to provide a simple, direct proof that the normalizer of the diagonal subalgebra is an inverse semigroup.  This type of approach was suggested by the referee of~\cite{reconstruct}, but the papers~\cite{BCvH2017,reconstruct,CarlsenSteinberg} use instead the full strength of their hypotheses to achieve this.  We introduce what we call the local bisection hypothesis on the normalizer of the diagonal; it is an analogue of the no non-trivial units hypothesis for group rings.  Under this hypothesis we show that the inverse semigroup of compact local bisections of the original ample groupoid can be recovered as the normalizer of the diagonal subalgebra, modulo its normal inverse subsemigroup of diagonal elements.  Since an ample groupoid is the tight groupoid of its associated inverse semigroup of compact local bisections~\cite{Exel,exelrecon,LawsonLenz}, this allows us to recover the groupoid.  A crucial advantage of working in this axiomatic setting is that we can show that the local bisection hypothesis can be checked on the interior of the isotropy bundle and hence we may reduce to the case of a group bundle, which is technically easier.  Also, we show that the local bisection hypothesis is invariant under diagonal-preserving ring isomorphisms, analogously to the no non-trivial units hypothesis for group rings~\cite[Chpt.~14, Thm~3.1]{PassmanBook}.

Our results, like those of~\cite{BCvH2017,reconstruct,CarlsenSteinberg}, are carried out  in the more general graded setting. This entails a little extra work because we must prove that a graded ample groupoid can be reconstructed from its graded inverse semigroup of homogeneous compact local bisections.  The ungraded case can be recovered from the graded case by assuming the group providing the grading is trivial.

The paper is organized as follows.  We begin with a section of preliminaries on group rings, inverse semigroups, groupoids and groupoids of germs.  We then introduce graded inverse semigroups and show how graded ample groupoids can be recovered from their graded inverse semigroup of homogeneous compact local bisections.  The next section studies the graded normalizer of the diagonal subalgebra of the algebra of a graded ample groupoid.  The normalizer is shown to be an inverse semigroup and its structure is elucidated.  We then introduce the local bisection hypothesis.  Sufficient conditions on a groupoid to ensure the local bisection hypothesis are provided. We then show that the graded groupoid can be reconstructed from its algebra and diagonal subalgebra under this hypothesis by recovering the inverse semigroup of homogeneous compact local bisections as the normalizer factored out by its normal subsemigroup of diagonal elements.   The paper ends with an application to Leavitt path algebras~\cite{LeavittBook}.

\section{Preliminaries}
Let $R$ be a commutative ring with unit.   The ring $R$ is \emph{indecomposable} (or \emph{connected}) if $0,1$ are the only idempotents of $R$.  This is equivalent to $R$ not being isomorphic to a non-trivial direct product and it is also equivalent to  the Zariski spectrum of $R$ being connected.  Every integral domain is indecomposable, but so are local rings like $\mathbb Z/p^n\mathbb Z$ with $p$ prime.  The ring $R$ is \emph{reduced} if it contains no non-zero nilpotent elements.   The coordinate ring of an affine variety that is connected, but not irreducible, gives a reduced indecomposable ring that is not an integral domain, e.g., $\mathbb C[x,y]/(xy)$. Another example is the integral group ring of $\mathbb Z/2\mathbb Z$.

The group of units of a ring $A$ will be denoted by $A^\times$.

\subsection{Group rings}
Let $G$ be a group and $RG$ the corresponding group algebra.   Put $R^\times G=\{rg\mid r\in R^\times, g\in G\}$.  Note that $R^\times G$ is a subgroup of the unit group $(RG)^\times$; the elements of $R^\times G$ are called \emph{trivial units}.  A group ring is said to have \emph{no non-trivial units} if $(RG)^\times=R^\times G$.  In this case, $G\cong (RG)^\times/R^\times$ and hence if $RG$ has no non-trivial units, then $RG\cong RH$ as $R$-algebras implies that $G\cong H$, cf.~\cite[Chpt.~14, Thm~3.1]{PassmanBook}.

A group is said to have the \emph{unique product property} if given two finite non-empty subsets $A,B$ of $G$, there is an element of $AB$ that can be uniquely written as a product $ab$ with $a\in A$ and $b\in B$.  This property is satisfied by any left or right orderable group and so, in particular, for any torsion-free abelian group or free group; see~\cite{PassmanBook}.  If $R$ is an integral domain and $G$ has the unique product property, then $RG$ has no zero divisors and no non-trivial units~\cite{PassmanBook}. We remark that it is known that if $R$ is an integral domain, then in order for $RG$ to have no zero divisors, $G$ must be torsion-free and that, for torsion-free groups, having no non-trivial units implies having no zero divisors; see~\cite{PassmanBook}. Kaplansky's Unit Conjecture asserts that if $G$ is torsion-free and $R$ is an integral domain, then $RG$ has no non-trivial units~\cite{PassmanBook}.

Let us say that $G$ is a \emph{trivial units only group} if $kG$ has no non-trivial units for every field $k$, for example, a unique product property group.  The following proposition is a direct consequence of~\cite[Theorem~3]{Neher} (or rather its proof, since the stated result is slightly weaker than we need).  We provide a proof for the reader's convenience.

\begin{Prop}\label{p:Neher}
Let $G\neq 1$ be a trivial units only group (e.g., a unique product property group). Then $RG$ has no non-trivial units if and only if $R$ is reduced and indecomposable.
\end{Prop}
\begin{proof}
We begin with necessity.  Fix $g\in G\setminus \{1\}$.  Suppose first that $R$ is not indecomposable and let $e\neq 0,1$ be an idempotent.  Then $e+(1-e)g$ is a non-trivial unit with inverse $e+(1-e)g\inv$.  If $0\neq n\in R$ is nilpotent, the $ng$ is nilpotent and so $1-ng$ is a non-trivial unit.

For sufficiency, we use a little algebraic geometry following~\cite{Neher}. Since $R$ is indecomposable, its prime spectrum $\mathrm{Spec}(R)$ is connected. If $\mathfrak p\in \mathrm{Spec}(R)$, then $\kappa(\mathfrak p)$ denotes the field of fractions of $R/\mathfrak p$ and if $a\in R$, then $a(\mathfrak p)$ denotes the image of $a$ in $\kappa(\mathfrak p)$.  Let $u\in RG$ be a unit, say $u=\sum_{g\in G}a_gg$.  The image of $u$ under the canonical homomorphism $RG\to \kappa(\mathfrak p)G$ is a unit and hence a trivial unit by hypothesis.  If follows that there is a unique element $h\in G$ with $a_h(\mathfrak p)\neq 0$.  Denote this element $g(\mathfrak p)$.  Thus we have a mapping $g\colon \mathrm{Spec}(R)\to G$ that we claim is locally constant.  Indeed, if $g(\mathfrak p)=h$, then $a_h\notin \mathfrak p$. Let $D(a_h)$ denote the basic open set of all prime ideals not containing $a_h$.  Then if $\mathfrak q\in D(a_h)$, we have $a_h(\mathfrak q)\neq 0$ and so $g(\mathfrak q)=h$.  It follows, since $\mathrm{Spec}(R)$ is connected, that $g$ is a constant mapping.  Therefore, there is an element $h\in G$ such that, for every  $\mathfrak  p\in \mathrm{Spec}(R)$, we have $a_h(\mathfrak p)\neq 0$ and $a_{h'}(\mathfrak p)=0$ for all $h'\neq h$.  Since $R$ is reduced, the intersection of all of its prime ideals is zero and so $a_{h'}=0$ for $h'\neq h$. On the other hand, since $a_h$ belongs to no prime ideal, and hence to no maximal ideal, we must have that $a_h$ is a unit.  Thus $u=a_hh$ is a trivial unit.
\end{proof}

For instance, $R\mathbb Z$ has no non-trivial units if and only if $R$ is reduced and indecomposable.  Note that the proof of necessity in Proposition~\ref{p:Neher} is valid for any non-trivial group.  Also in the proof of sufficiency, we just need that $kG$ has no non-trivial units for each field $k$ that is an $R$-algebra.

Higman~\cite{higmanunit} showed that if $G$ is a finite abelian group of exponent $4$ or $6$, or a quaternion group, then $\mathbb ZG$ has no non-trivial units.  He also showed that if $\mathbb ZG$ has no non-trivial units, then $\mathbb Z[G\times \mathbb Z/2\mathbb Z]$ has only trivial units.

\subsection{Inverse semigroups}
An \emph{inverse semigroup} is a semigroup $S$ such that, for each $s\in S$, there exists a unique element $s^*\in S$ with $ss^*s=s$ and $s^*ss^*=s^*$.  Moreover, one has that $(s^*)^*=s$, $(st)^*=t^*s^*$ and $ss^*t^*t= t^*tss^*$ for all $s,t\in S$.  The reader is referred to Lawson's book~\cite{Lawson} for the theory of inverse semigroups.  Note that every group is an inverse semigroup.

The set $E(S)$ of idempotents of an inverse semigroup is a commutative subsemigroup and is a meet semilattice with respect to the partial ordering $e\leq f$ if $ef=e$; the meet is the product.  Also, if $e\in E(S)$, then $e^*=e$.  The partial order extends to the entire inverse semigroup by putting $s\leq t$ if $s=te$ for some idempotent $e\in E(S)$ (or, equivalently, $s=ft$ for some  $f\in E(S)$).  This partial order is stable for multiplication and inversion.

Inverse semigroups can be alternatively axiomatized as those von Neumann regular semigroups with commuting idempotents where we recall that a semigroup $S$ is \emph{von Neumann regular} if, for all $s\in S$, there exists $s'\in S$ with $ss's=s$.

A \emph{zero element} of an inverse semigroup $S$ is an element $z$ such that $zx=z=xz$ for all $x\in S$.  Zero elements are unique when they exist and are usually denoted $0$.  Most inverse semigroups of interest to us will have a zero.

Any semigroup homomorphism $\p\colon S\to T$ of inverse semigroups automatically preserves the involution.

An inverse subsemigroup $T$ of $S$ is said to be \emph{full} if $E(T)=E(S)$.  Notice that if $T$ is a full inverse subsemigroup of $S$, then it is an \emph{order ideal}, that is, if $s\leq t$ with $t\in T$, then $s\in T$.  Indeed, $s=te$ with $e\in E(S)=E(T)$ and so $s\in T$.

\subsection{Ample groupoids}
A \emph{groupoid} $\mathscr G$ is a small category in which every morphism is an isomorphism.  We write $\mathscr G\skel 0$ for the set of objects of $\mathscr G$ and $\mathscr G\skel 1$ for the set of arrows.  The domain and range maps are denoted $\dom$ and $\ran$, respectively.  We typically view $\mathscr G\skel 0$ as a subset of $\mathscr G\skel 1$ by identifying an object $x$ with the identity morphism at $x$.  A \emph{topological groupoid} is a groupoid in which $\mathscr G\skel 0$ and $\mathscr G\skel 1$ are equipped with topologies such that all the groupoid structure maps are continuous.  A topological groupoid is \emph{\'etale} if all the structure maps are local homeomorphisms; actually, it is enough for the domain map to be a local homeomorphism. In this case $\mathscr G\skel 0$ is an open subspace of $\mathscr G\skel 1$.  See~\cite{Exel,Paterson,resendeetale} for details.  Morphisms of topological groupoids are continuous functors.  An isomorphism of topological groupoids is then a continuous functor with a continuous inverse.

We shall call a Hausdorff space $X$ with  a basis of compact open sets a \emph{Boolean space}.  Note that $X$ is determined up to homeomorphism by its (generalized) Boolean algebra of compact open subsets by Stone duality.  An \'etale groupoid $\mathscr G$ is called \emph{ample} if $\mathscr G\skel 0$ is a Boolean space.  If $\mathscr G\skel 1$ is Hausdorff, then $\mathscr G\skel 1$ will also be a Boolean space.  In this paper we assume that all groupoids are Hausdorff although we shall repeat this hypothesis in the main theorems.

Groups are precisely one-object ample groupoids.  More generally, any discrete groupoid is ample.  If $X$ is a Boolean space, then we can view it as an ample groupoid in which $\mathscr G\skel 0=X=\mathscr G\skel 1$, that is, as a groupoid of identity morphisms.

A \emph{local bisection} of an \'etale groupoid $\mathscr G$ is an open subset $U$ of $\mathscr G\skel 1$ such that $\dom|_U$ and $\ran|_U$ are injective.    The set of  compact local bisections is denoted $\Bis_c(\mathscr G)$.  The groupoid $\mathscr G$ is ample if and only if $\Bis_c(\mathscr G)$ is a basis for the topology on $\mathscr G\skel 1$ (cf.~\cite{Paterson}).  If $U,V\in \Bis_c(\mathscr G)$, then so is
 \[UV=\{\alpha\beta\mid \alpha\in U,\beta\in V\}\] and so $\Bis_c(\mathscr G)$ is an inverse semigroup with $U\inv =\{\gamma\inv\mid \gamma\in U\}$ as the inverse of $U$ (in the sense of inverse semigroup theory)~\cite{Exel,Paterson}.

 If $x\in \mathscr G\skel 0$, then the \emph{isotropy group} of $\mathscr G$ at $x$ is the group \[G_x=\{\gamma\in \mathscr G\skel 1\mid \dom(\gamma)=x=\ran(\gamma)\}.\]  The \emph{orbit} of $x$ is the set of objects $y$ such that there is a morphism from $x$ to $y$.  A subset of $\mathscr G\skel 0$ is \emph{invariant} if it is a union of orbits.

  If $\mathscr G$ is a groupoid, the \emph{isotropy bundle} of $\mathscr G$ is the subgroupoid consisting of all objects and those arrows belonging to $\bigcup_{x\in \mathscr G\skel 0}G_x$.  The isotropy bundle of an \'etale groupoid need not be \'etale but the interior of the isotropy bundle is \'etale.  An \'etale groupoid is said to be \emph{effective} if the interior of the isotropy bundle consists of just the identity morphisms.  This is equivalent to the natural action of $\Bis_c(\mathscr G)$ on $\mathscr G\skel 0$ being faithful (cf.~\cite{Exel}).  A closely related notion is that of a topologically principal groupoid.  An \'etale groupoid is \emph{topologically principal} if the set of objects with trivial isotropy group is dense.  Topologically principal ample groupoids are effective and the converse is true for second countable groupoids~\cite{operatorsimple1,renaultcartan}.  In general, these notions are different.  In~\cite{operatorsimple1}, an effective ample groupoid is constructed in which every isotropy group is infinite cyclic.

If $G$ is a group, a \emph{cocycle} $c\colon \mathscr G\to G$ is a continuous functor (where $G$ is viewed as a discrete one-object groupoid).  The fiber over $g\in G$ is denoted $\mathscr G_g$ and is called the \emph{homogeneous component of $g$}.  Note that $\mathscr G_1$ is a clopen subgroupoid of $\mathscr G$.  We call $\mathscr G$ a $G$-graded groupoid.  There is an obvious category of $G$-graded groupoids. We call a local bisection $U$ \emph{homogeneous} if $U\subseteq c\inv(g)$ for some $g\in G$.  The homogeneous compact local bisections form an inverse subsemigroup $\Bis_c^h(\mathscr G)$ of $\Bis_c(\mathscr G)$.  Of course, every groupoid is trivially graded by the trivial group, in which case, $\Bis_c^h(\mathscr G)$ reduces to $\Bis_c(\mathscr G)$.

\subsection{Groupoids of germs}
Let $X$ be a Boolean space. We denote by $I_X$ the inverse semigroup of all partial homeomorphisms of $X$ with compact open domain. An action of an inverse semigroup $S$ with zero on $X$ is a zero-preserving homomorphism $\p\colon S\to I_X$ such that if $X_e$ denotes the domain of an idempotent $e$, then $\bigcup_{e\in E(S)}X_e=X$; this last condition says that the action is non-degenerate. An important example is given by the spectral action of $S$, which we now proceed to describe.

If $E$ is a meet semilattice with zero, a \emph{character} of $E$ is a non-zero homomorphism $\tau\colon E\to \{0,1\}$ of semilattices with zero.  The space $\mathrm{Spec}(E)$ of all characters on $E$ is a Boolean space with respect to the topology of pointwise convergence.  An \emph{ultracharacter} is a maximal character with respect to the pointwise ordering on characters.   The set of ultracharacters of $E$ is denoted $\wh E$.  It is well known that if $E$ is a generalized Boolean algebra, then the ultracharacters are precisely the generalized Boolean algebra homomorphisms.  In this case, $\wh E$ is a closed subspace of $\mathrm{Spec}(E)$ with basic compact open sets of the form $D(e)=\{\tau\in \wh{E}\mid \tau(e)=1\}$.  If $X$ is a Boolean space and $\mathcal B$ is its generalized Boolean algebra of compact open sets, then $X$ is homeomorphic to $\wh B$ via the mapping $x\mapsto \tau_x$ where $\tau_x(U) = \chi_U(x)$ (here $\chi_U$ is the characteristic function of $U$).

If $S$ is an inverse semigroup with zero, then $S$ acts on $\mathrm{Spec}(E(S))$ as follows. For $e\in E(S)$, let $D(e) = \{\tau\in \mathrm{Spec}(E(S))\mid \tau(e)=1\}$.  Then the domain of the action of $s\in S$ is $D(s^*s)$ and the range is $D(ss^*)$.  The action is given by $s\tau(e) = \tau(s^*es)$. The space $\wh{E(S)}$ is invariant under the action of $S$ and so if $E(S)$ is a generalized Boolean algebra, then $S$ acts on $\wh{E(S)}$.  See~\cite{Exel} for details.

If $S$ acts on a Boolean space $X$, the \emph{groupoid of germs} $\mathscr G=S\ltimes X$ of the action is described as follows.  The object space is $X$.  The arrow space $\mathscr G\skel 1$ consists of all equivalence classes of pairs $(s,x)$ with $x\in X_{s^*s}$ where $(s,x)\sim (t,y)$ if $x=y$ and there exists $u\in S$ with $x\in X_{u^*u}$ and $u\leq s,t$.  In other words, $s$ and $t$ have a common restriction belonging to $S$ and defined at $x$.  The class of $(s,x)$  is denoted by $[s,x]$.  A basis for the topology on $\mathscr G\skel 1$ is given by the sets
\[(s,U)=\{[s,x]\mid x\in U\}\] where $U\subseteq X_{s^*s}$.  The domain and range maps are given by $\dom([s,x])=x$ and $\ran([s,x])=sx$.  The product is given by $[s,tx][t,x]=[st,x]$ and the inverse is given by $[s,x]\inv = [s^*,sx]$.  The identity at $x$ is $[e,x]$ where $e$ is any idempotent with $x\in X_e$.  We note that it is not always the case that $S\ltimes X$ is Hausdorff.  The reader should consult~\cite{Exel} for details.  The groupoid $S\ltimes \mathrm{Spec}(S)$ is called the \emph{universal groupoid} of $S$.  Its algebra is isomorphic to the semigroup algebra of $S$~\cite{mygroupoidalgebra}.

If $\mathscr G$ is an ample groupoid, then it follows from~\cite[Prop.~5.4]{Exel} and Stone duality that $\mathscr G\cong \Bis_c(\mathscr G)\ltimes \wh{E(\Bis_c(\mathscr G))}$. See, for example,~\cite[Thm~4.8]{exelrecon} or~\cite{LawsonLenz} for details.  Consequently, ample groupoids $\mathscr G$ and $\mathscr G'$ are isomorphic if and only if the inverse semigroups $\Bis_c(\mathscr G)$ and $\Bis_c(\mathscr G')$ are isomorphic.

\subsection{Groupoid algebras}
If $\mathscr G$ is a Hausdorff ample groupoid and $R$ is a commutative ring with unit, then the \emph{groupoid algebra} $R\mathscr G$ has underlying $R$-module the space $C_c(\mathscr G\skel 1,R)$ of compactly supported locally constant mappings from $\mathscr G\skel 1$ to $R$.  The product is given by convolution
\[f\ast g(\gamma)= \sum_{\dom(\alpha)=\dom(\gamma)}f(\gamma\alpha\inv)g(\alpha).\]
Details can be found in~\cite{mygroupoidalgebra}.  In this paper, we shall write $fg$ as shorthand for $f\ast g$.  Note that $R\mathscr G$ is spanned by the characteristic functions of compact local bisections and if $U,V\in \Bis_c(\mathscr G)$, then $\chi_U\chi_V=\chi_{UV}$.  In fact, $R\mathscr G$ is the quotient of the semigroup algebra $R\Bis_c(\mathscr G)$ by the relations $\chi_U+\chi_V=\chi_{U\cup V}$ whenever $U,V$ are disjoint compact open subsets of $\mathscr G\skel 0$~\cite{mygroupoidarxiv,operatorguys1}.

If $X$ is a Boolean space, viewed as an ample groupoid consisting of identity morphisms, then the groupoid algebra is $C_c(X,R)$ with pointwise operations.  If we view a group $G$ as a one-object ample groupoid, then $RG$ is the usual group algebra.

If $c\colon \mathscr G\to G$ is a cocycle, then $R\mathscr G$ is a $G$-graded algebra where the homogeneous component $R\mathscr G_g$ of $g\in G$ consists of those mappings whose support is contained in the clopen set $\mathscr G_g$~\cite{operatorguys1}.

If $\mathscr H$ is an open subgroupoid of $\mathscr G$, then $R\mathscr H$ is a subalgebra of $R\mathscr G$ in a natural way (by extending a compactly supported mapping on $\mathscr H\skel 1$ to $\mathscr G\skel 1$ by $0$ at all undefined morphisms).  In particular $C_c(\mathscr G\skel 0,R)$ is naturally a subalgebra of $R\mathscr G$ that we call the \emph{diagonal subalgebra} and denote by $D_R(\mathscr G)$ or, if $R$ is understood, just $D(\mathscr G)$.   Note that in the graded setting, the diagonal subalgebra is contained in the homogeneous component of $1$.  If $H$ is a group, then the diagonal subalgebra of $RH$ consists of the scalar multiples of the identity.

An isomorphism $\Phi\colon R\mathscr G\to R\mathscr G'$ of groupoid algebras is called \emph{diagonal-preserving} if $\Phi(D(\mathscr G))=D(\mathscr G')$.  Note that if  $\mathscr G,\mathscr G'$ are $G$-graded groupoids and there is a graded isomorphism $\p\colon \mathscr G\to \mathscr G'$, then there is a diagonal-preserving graded isomorphism $R\mathscr G\to R\mathscr G'$ of $R$-algebras mapping $f$ to $f\circ \pinv$.

If $X$ is a closed invariant subspace of $\mathscr G\skel 0$ and $\mathscr G|_X$ is the (closed)  full subgroupoid of $\mathscr G$ with object set $X$, then restriction induces a homomorphism $R\mathscr G\to R\mathscr G|_X$.

\begin{Prop}\label{p:z.of.diag}
Let $\mathscr G$ be a Hausdorff ample groupoid and $R$ a commutative ring with unit.  Let $\mathscr H$ be the interior of the isotropy bundle of $\mathscr G$.  Then the centralizer of the diagonal subalgebra $D_R(\mathscr G)$ is $R\mathscr H$ (viewed as the subalgebra of $R\mathscr G$ consisting of those functions supported on $\mathscr H$).
\end{Prop}
\begin{proof}
Suppose that $f\in R\mathscr H$ and $g\in D_R(\mathscr G)$. Then $fg(\alpha) = f(\alpha)g(\dom(\alpha))=g(\dom(\alpha))f(\alpha)$.  If $f(\alpha)\neq 0$, then $\alpha\in \mathscr H\skel 1$ and so $\ran(\alpha)=\dom(\alpha)$ and also $g(\dom(\alpha))f(\alpha)=g(\ran(\alpha))f(\alpha) = gf(\alpha)$. If $f(\alpha)=0$, then $gf(\alpha) = g(\ran(\alpha))f(\alpha)=0$.  So $fg(\alpha)=gf(\alpha)$ in all cases.

Conversely, if $f\notin R\mathscr H$, then there exists $\alpha\in \supp(f)$ with $\dom(\alpha)\neq \ran(\alpha)$.  Then we can find $U\subseteq \mathscr G\skel 0$ compact open with $\dom(\alpha)\in U$ and $\ran(\alpha)\notin U$.  Then $f\chi_U(\alpha) = f(\alpha)\neq 0$ and $\chi_Uf(\alpha) =\chi_U(\ran(\alpha))f(\alpha)=0$.  Thus $f$ does not centralize $D_R(\mathscr G)$.  This completes the proof.
\end{proof}

Since any isomorphism sends centralizers to centralizers, we obtain our first corollary of Proposition~\ref{p:z.of.diag}.

\begin{Cor}\label{c:isotropy.pres}
Let $\mathscr G$ and $\mathscr G'$ be ample Hausdorff groupoids and $R$ a commutative ring with unit.  If $\Phi\colon R\mathscr G\to R\mathscr G'$ is a diagonal-preserving ring isomorphism, then it restricts to a diagonal-preserving ring isomorphism of the algebra of the interior of the  isotropy bundle of $\mathscr G$ with that of the interior of the isotropy bundle of $\mathscr G'$.
\end{Cor}

The implication (1) implies (2) of the following corollary was proved in~\cite[Proposition~3.8]{groupoidprimitive}, although it should probably be considered folklore.

\begin{Cor}\label{c:effective.diag.det}
Let $\mathscr G$ be a Hausdorff ample groupoid and $R$ a commutative ring with unit.  Then the following are equivalent.
\begin{enumerate}
\item $\mathscr G$ is effective.
\item $D_R(\mathscr G)$ is a maximal commutative subring of $R\mathscr G$.
\item $D_R(\mathscr G)$ is its own centralizer in $R\mathscr G$.
\end{enumerate}
\end{Cor}
\begin{proof}
If $\mathscr G$ is effective, then $D_R(\mathscr G)$ is its own centralizer by Proposition~\ref{p:z.of.diag}.  If $D_R(\mathscr G)$ is its own centralizer, then clearly it is a maximal commutative subring.  Suppose that $D_R(\mathscr G)$ is a maximal commutative subring and let suppose that $\gamma\in \mathscr G\skel 1$ belongs to the interior of the isotropy bundle.  Then there is a compact open neighborhood $U$ of $\gamma$ contained in the interior of the isotropy bundle.  But then $\chi_U$ centralizes $D_R(\mathscr G)$ by Proposition~\ref{p:z.of.diag} and so the subring $D_R(\mathscr G)[\chi_U]$ generated by $\chi_U$ and $D_R(\mathscr G)$ is a commutative subring containing $D_R(\mathscr G)$.  By maximality, we conclude that $\chi_U\in D_R(\mathscr G)$ and hence $\gamma\in \mathscr G\skel 0$.  Thus $\mathscr G$ is effective.
\end{proof}

An immediate consequence of Corollary~\ref{c:effective.diag.det} is our next corollary.

\begin{Cor}\label{c:eff.iff}
Let $\mathscr G,\mathscr G'$ be Hausdorff ample groupoids and let $R$ be a commutative ring.  Suppose that $\Phi\colon R\mathscr G\to R\mathscr G'$ is a diagonal-preserving ring isomorphism.  Then $\mathscr G$ is effective if and only if $\mathscr G'$ is effective.
\end{Cor}

\subsection{Distributive inverse semigroups}\label{ss:distributive}
Let $S$  be an inverse semigroup.  Then $s,t\in S$ are said to be \emph{compatible}~\cite{Lawson} if $st^*,t^*s\in E(S)$.  Notice that if $s,t$ have a common upper bound in $S$, then they are compatible.  A subset of $S$  is \emph{compatible} if each pair of its elements is compatible.  Compatible subsets are precisely those which can potentially have a join.

An inverse semigroup $S$ is \emph{distributive} if it admits joins of compatible pairs and products in $S$ distribute over joins.  It is easy to see that if an inverse semigroup is distributive, then it admits joins of any finite compatible set.
Of course, any isomorphism of distributive inverse semigroups preserves the join and the class of distributive inverse semigroups is closed under isomorphism.  Also, note that any homomorphism of inverse semigroups preserves the compatibility relation. Details can be found, for example, in~\cite{LawsonLenz}.

Let $\mathscr G$ be a Hausdorff ample groupoid.  The reader should check that $U,V\in \Bis_c(\mathscr G)$ are compatible if and only if $U\cup V$ is a local bisection. Hence $\Bis_c(\mathscr G)$ is a distributive inverse semigroup with union as the join.

If $S$ is a distributive inverse semigroup and $T\leq S$ is a full inverse subsemigroup which is also distributive, then $T$ is closed under all joins of finite compatible subsets.  Indeed, if $A\subseteq T$ is finite and compatible and $s,t$ are the joins of $A$ in $S$ and $T$, respectively, then $s\leq t$ and so $s\in T$ because $T$ is an order ideal.  Thus $s=t$.

\section{Graded inverse semigroups and groupoids of germs}

Let $S$ be an inverse semigroup with zero and $G$ a group.  A \emph{partial homomorphism} $\theta\colon S\to G$ is a mapping $\theta\colon S\setminus \{0\}\to G$ (abusing notation) such that $\theta(st)=\theta(s)\theta(t)$ whenever $st\neq 0$.

\begin{Prop}\label{p:partial.group}
Let $\theta\colon S\to G$ be a partial homomorphism.
\begin{enumerate}
	\item $\theta(e)=1$ for all non-zero idempotents $e$.
	\item $\theta(s^*)=\theta(s)\inv$ for $s\in S\setminus \{0\}$.
	\item $0\neq s\leq t$ implies $\theta(s)=\theta(t)$.
\end{enumerate}
\end{Prop}
\begin{proof}
If $0\neq e=e^2$, then $\theta(e)=\theta(e)\theta(e)$ and so $\theta(e)=1$.  Hence, if $s\neq 0$, then $\theta(s)\theta(s^*)=\theta(ss^*)=1$ and so $\theta(s^*)=\theta(s)\inv$.  Finally, if $0\neq s\leq t$, then $s=te$ with $0\neq e$ idempotent and so $\theta(s)=\theta(t)\theta(e)=\theta(t)$.  This completes the proof.
\end{proof}

A \emph{$G$-graded inverse semigroup} is an inverse semigroup equipped with a partial homomorphism $\theta\colon S\to G$. For $g\in G$, put $S_g=\theta\inv(g)\cup \{0\}$.  We sometimes call $\theta$ a \emph{grading} of $S$.  A \emph{graded homomorphism} of $G$-graded inverse semigroups is a zero-preserving homomorphism respecting the grading.

Our main example of a $G$-graded inverse semigroup comes from a $G$-graded ample groupoid $c\colon \mathscr G\to G$.  Then the inverse semigroup of homogeneous compact local bisections $\Bis_c^h(\mathscr G)$ is a $G$-graded inverse semigroup with respect to the obvious grading induced by $c$.

Note that any inverse subsemigroup with zero of a $G$-graded inverse semigroup inherits a $G$-grading.

If $S$ is a $G$-graded inverse semigroup acting on a Boolean space $X$, then the groupoid of germs $\mathscr G=S\ltimes X$ is naturally $G$-graded via the cocycle $c([s,x])=\theta(s)$ where $\theta$ is the grading of $S$.  This is well defined because if $[s,x]=[t,x]$, then there exists $0\neq u\leq s,t$ and so $\theta(s)=\theta(u)=\theta(t)$ by Proposition~\ref{p:partial.group}.  It is a cocycle since  $[s,tx][t,x]=[st,x]$ and $\theta(st)=\theta(s)\theta(t)$ (here we are using that $0$ maps to the empty map and hence $st\neq 0$).  Continuity follows because $c\inv(g) = \bigcup_{s\in \theta\inv(g)}(s,X_{s^*s})$, which is open.

We shall require a condition that guarantees that if $T\leq S$ is an inverse subsemigroup, then the groupoid of germs for the action of $S$ on $X$ is isomorphic to the groupoid of germs for the restricted action of $T$ on $X$.

If $\p\colon S\to I_X$ is an action on a Boolean space, we say that $T\leq S$ is \emph{cofinal} with respect to the action if, for all $x\in X$ and $s\in S$ with $x\in X_{s^*s}$, there exists $t\in T$ such that $t\leq s$ and $x\in X_{t^*t}$.

\begin{Prop}\label{p:full.cofinal}
Let $S$ be a $G$-graded inverse semigroup acting on a Bool\-e\-an space $X$.  Suppose that $T\leq S$ is a full inverse subsemigroup that is  cofinal with respect to the action.  Then the groupoids of germs $T\ltimes X$ and $S\ltimes X$ are isomorphic as $G$-graded group\-oids.
\end{Prop}
\begin{proof}
Note that since $E(S)=E(T)$, the restriction of the action to $T$ is non-degenerate.
We define a functor $\rho\colon T\ltimes X\to S\ltimes X$ by the identity on objects and by $\rho([t,x]_T)=[t,x]_S$ where we use $[u,y]_R$ to denote the germ class of $(u,y)$ for the inverse semigroup $R$.  The mapping $\rho$ is clearly well defined.  To see that $\rho$ is injective, suppose $[t,x]_S=[t',x]_S$ with $t,t'\in T$.  Then there exists $u\in S$ with $u\leq t,t'$  and $x\in X_{u^*u}$.  But $T$ is an order ideal, so $u\in T$ and hence $[t,x]_T=[t',x]_T$.  Also, $\rho$ is surjective since if $[s,x]_S$ is an arrow of $S\ltimes X$, then by cofinality, there exists $t\in T$ with $t\leq s$ and $x\in X_{t^*t}$.  Thus $[s,x]_S=[t,x]_S=\rho([t,x]_T)$ and so $\rho$ is onto.  Clearly $\rho$ is a functor.  It is continuous because $\rho\inv(s,U) = \bigcup_{t\in T, t\leq s}(t,U\cap X_{t^*t})$.  It is open because $\rho(t,U) = (t,U)$ for $U\subseteq X_{t^*t}$.  Clearly $\rho$ preserves the grading.  This completes the proof.
\end{proof}

As a consequence, we show that if $\mathscr G$ is a $G$-graded ample groupoid, then $\mathscr G\cong \Bis_c^h(\mathscr G)\ltimes \wh{E(\Bis_c^h(\mathscr G))}$ as $G$-graded groupoids.

\begin{Thm}\label{t:graded.reconstruct.inv.sgp}
Let $c\colon \mathscr G\to G$ be a $G$-graded ample groupoid.  Then $\mathscr G\cong \Bis_c^h(\mathscr G)\ltimes \wh{E(\Bis_c^h(\mathscr G))}$ as $G$-graded groupoids.
\end{Thm}
\begin{proof}
As mentioned earlier $\mathscr G\cong \Bis_c(\mathscr G)\ltimes \wh{E(\Bis_c(\mathscr G))}$; see~\cite[Prop.~5.4]{Exel} and~\cite[Thm~4.8]{exelrecon}.  The isomorphism takes an arrow $\gamma\colon x\to y$ to $[U,\tau_x]$ where $U$ is any compact local bisection containing $\gamma$. Since the identities of $\mathscr G$ all belong to $\mathscr G_1$, it follows that $\Bis_c^h(\mathscr G)$ is a full inverse subsemigroup of $\Bis_c(\mathscr G)$.   To see that it is cofinal with respect to the action on $\wh{E(\Bis_c(\mathscr G))}$, let $x\in \mathscr G\skel 0$ and $U\in \Bis_c(\mathscr G)$ with $x\in \dom(U)$, i.e., $\tau_x(U\inv U)=1$. Let $\gamma\in U$ with $\dom(\gamma)=x$ and put $g=c(\gamma)$.  Then, as $c\inv(g)$ is open, there is a compact local bisection $V$ with $\gamma\in V\subseteq U$ and $V\subseteq c\inv(g)$.  Then $V\in \Bis_c^h(\mathscr G)$, $V\leq U$ and $x\in \dom(V)$ (i.e., $\tau_x(V)=1$).  Thus $\Bis^h_c(\mathscr G)$ is cofinal in $\Bis_c(\mathscr G)$ with respect to the action and so $\Bis_c(\mathscr G)\ltimes \wh{E(\Bis_c(\mathscr G))}\cong \Bis_c^h(\mathscr G)\ltimes \wh{E(\Bis_c^h(\mathscr G))}$ by Proposition~\ref{p:full.cofinal}.  It remains to show that the isomorphism is $G$-graded.

Indeed, if $\gamma\colon x\to y$ with $c(\gamma)=g$, then the above argument shows that we can find $V\in \Bis^h_c(\mathscr G)$ with $\gamma\in V$ and $V\subseteq c\inv(g)$.  Then $\gamma\mapsto [V,\tau_x]$ under the isomorphism $\mathscr G\to \Bis_c^h(\mathscr G)\ltimes \wh{E(\Bis_c^h(\mathscr G))}$ (using the proof of Proposition~\ref{p:full.cofinal}).  It follows that the isomorphism preserves the grading.
\end{proof}

Consequently, to reconstruct $\mathscr G$ as a $G$-graded groupoid, it suffices to reconstruct $\Bis_c^h(\mathscr G)$ as a $G$-graded inverse semigroup.

\section{The normalizer of the diagonal}
In this section, we generalize the notion of having no non-trivial units from group algebras to ample groupoid algebras.

\subsection{The local bisection hypothesis}\label{ss:local.bis}
Let $R$ be a commutative ring with unit and $\mathscr G$ a Hausdorff ample $G$-graded groupoid with cocycle $c\colon \mathscr G\to G$.     For $g\in G$, let
\begin{equation}\label{eq:normalizer}
\begin{split}
N_g = \{m\in R\mathscr G_g\mid\,\, & \exists m'\in R\mathscr G_{g\inv}, mm'm=m, m'mm'=m',  \\ & mD(\mathscr G)m'\cup m'D(\mathscr G)m\subseteq D(\mathscr G)\}.
\end{split}
\end{equation}

We call $N=\bigcup_{g\in G} N_g$ the \emph{(graded) normalizer} of $D(\mathscr G)$.  Observe that if $m\in N_g$ and $m'$ is as in \eqref{eq:normalizer}, then $m'\in N_{g\inv}$.  Trivially, if $U\subseteq \mathscr G_g$ is a homogeneous compact local bisection, then $\chi_U\in N_g$ with $\chi_U'=\chi_{U\inv}$.   Our first goal is to show that $N$ is a $G$-graded inverse semigroup.

\begin{Prop}\label{p:are.idem}
If $m\in N_g$ and $m'$ is as in \eqref{eq:normalizer}, then $m'm,mm'$ are idempotents of $D(\mathscr G)$.
\end{Prop}
\begin{proof}
They are clearly idempotents.  Let $K=\dom(\supp(m'))\cup \ran(\supp(m))$.  Note that $\chi_K\in D(\mathscr G)$ and so $m'm=m'\chi_K m\in D(\mathscr G)$. Similarly, $mm'\in D(\mathscr G)$.
\end{proof}

Recall that a semigroup $S$ is an inverse semigroup if and only if it is von Neumann regular with commuting idempotents.

\begin{Prop}\label{p:is.inverse}
We have that $N$ is a $G$-graded inverse semigroup, where the grading sends a non-zero element of $N_g$ to $g$.
\end{Prop}
\begin{proof}
Let $m\in N_g$ and $n\in N_h$ and suppose that $m'\in N_{g\inv}$ and $n'\in N_{h\inv}$ are as in \eqref{eq:normalizer}.
Note that $m'mnn'=nn'm'm$  by Proposition~\ref{p:are.idem} and commutativity of $D(\mathscr G)$. We then compute $mnn'm'mn= mm'mnn'n=mn$ and similarly  $n'm'mnn'm'=n'nn'm'mm'=n'm'$.     Also, $mnD(\mathscr G)n'm'\cup n'm'D(\mathscr G)mn\subseteq D(\mathscr G)$ and so $mn\in N$ and we can take $(mn)' = n'm'$.  It follows that $N$ is a von Neumann regular semigroup.  If $e\in N_g$ is an idempotent and $e'\in N_{g\inv}$ is as in \eqref{eq:normalizer}, then  we get that $e'=e'ee'=e'eee'=ee'e'e=ee(e'e')ee=ee=e$ where the third equality is from Proposition~\ref{p:are.idem} and commutativity of $D(\mathscr G)$ and the penultimate equality is from the previous computation with $m=e=n$.  Therefore, $e=ee=ee'\in D(\mathscr G)$ by Proposition~\ref{p:are.idem}. Thus $E(N)\subseteq D(\mathscr G)$, which is commutative, and so the idempotents of $N$ commute, establishing that $N$ is an inverse semigroup.  Clearly, mapping $m\in N_g\setminus \{0\}$ to $g$ is a partial homomorphism and so $N$ is $G$-graded.
\end{proof}

\begin{Rmk}\label{r:where.the.idems.are}
Note that the above proof shows that each idempotent $e$ of $N$ belongs to $D(\mathscr G)$. We shall always view $\Bis_c^h(\mathscr G)$ as a subsemigroup of $N$ via $U\mapsto \chi_U$.
\end{Rmk}

To understand the idempotents of $N$, we from now on impose the assumption that $R$ is indecomposable. Let $\mathcal B$ be the Boolean algebra of compact open subsets of $\mathscr G\skel 0$. Note that we can canonically identify $\mathscr G\skel 0$ with the Stone dual $\wh B$ of $\mathcal B$.

\begin{Prop}\label{p:idempotents}
If $R$ is indecomposable, then \[E(D(\mathscr G)) = E(N)=\{\chi_U\mid U\in \mathcal B\}.\]
\end{Prop}
\begin{proof}
This is immediate from Remark~\ref{r:where.the.idems.are} since the product in $D(\mathscr G)$ is pointwise and $R$ contains only the idempotents $0$ and $1$.
\end{proof}

For an indecomposable $R$, we can further clarify the nature of the idempotents $m'm$ and $mm'$.

\begin{Prop}\label{p:the.idems}
Suppose that $R$ is indecomposable and let $m,m'\in N_g$.  Then $m'm=\chi_{\dom(\supp(m))}$ and $mm' = \chi_{\ran(\supp(m))}$.
\end{Prop}
\begin{proof}
As $m'm$ is an idempotent, it follows from Proposition~\ref{p:idempotents} that $m'm=\chi_U$ for some compact open $U\subseteq \mathscr G\skel 0$.  So it suffices to show the equality $\supp(m'm) = \dom(\supp(m))$. If $m(\alpha)\neq 0$, then from $m(\alpha) = m(m'm)(\alpha) = m\chi_U(\alpha)=m(\alpha)\chi_U(\dom(\alpha))$, we deduce that $\dom(\alpha)\in U$.  If $x\in U$, then $1=m'm(x) = \sum_{\dom(\alpha)=x}m'(\alpha\inv)m(\alpha)$ and so $m(\alpha)\neq 0$ for some $\alpha$ in $\dom\inv(x)$.  Thus $x\in \dom(\supp(m))$.  The other equality is dual.
\end{proof}

A consequence of Proposition~\ref{p:the.idems} is that intersection of the diagonal subalgebra with $N$ is an inverse subsemigroup.

\begin{Cor}\label{c:diag.inv.sub}
Let $R$ be indecomposable.  Then, for $m\in N\cap D(\mathscr G)$, we have that $m'\in N\cap D(\mathscr G)$ where $m'$ is as in \eqref{eq:normalizer}.
\end{Cor}
\begin{proof}
Let $\alpha\in \supp(m')$ and say $x=\dom(\alpha)$.  By Proposition~\ref{p:the.idems}, we have that
$1=mm'(x)=m(x)m'(x)$, as $m$ is supported on $\mathscr G\skel 0$, and hence $m(x)\in R^\times$.  But then
$m'm(\alpha) =m'(\alpha)m(x)\neq 0$ because $m$ is supported on $\mathscr G\skel 0$, $\alpha\in \supp(m')$ and $m(x)$ is a unit.  We deduce that $\alpha\in \mathscr G\skel 0$ by Proposition~\ref{p:the.idems} and hence $m'\in N\cap D(\mathscr G)$.
\end{proof}

Our next result says that the support of an element of the normalizer is highly restricted, and close to being a local bisection.

\begin{Prop}\label{p:isotropy}
Let $R$ be indecomposable.  If $m\in N$ and $\alpha,\beta\in \supp(m)$, then $\dom(\alpha)=\dom(\beta)$ if and only if $\ran(\alpha)=\ran(\beta)$.
\end{Prop}
\begin{proof}
Let $x=\dom(\alpha)=\dom(\beta)$.  Let $y=\ran(\alpha)$ and $z=\ran(\beta)$ and suppose that $y\neq z$.  Then, since $\dom\inv(x)\cap \supp(m)$ is finite and $\mathscr G\skel 0$ is Hausdorff, we can find disjoint compact open subsets $U,V\subseteq \mathscr G\skel 0$ with $U\cap \ran(\dom\inv(x)\cap \supp(m)) = \{y\}$ and $V\cap \ran(\dom\inv(x)\cap \supp(m)) =\{z\}$.  Then $m_U=\chi_Um$, $m_V=\chi_Vm$ and $m_{U\cup V} = \chi_{U\cup V}m$ all belong to $N$.  Note that $m_U' = m'\chi_U$, $m_V' = m'\chi_V$ and $m_{U\cup V}' = m'\chi_{U\cup V}$.  We conclude from Proposition~\ref{p:the.idems} that $1=m_U'm_U(x)=m_V'm_V(x)=m_{U\cup V}'m_{U\cup V}(x)$ as $m_U(\alpha)=m(\alpha)\neq 0$, $m_V(\beta)=m(\beta) \neq 0$ and $m_{U\cup V}(\alpha)=m(\alpha)\neq 0$.  Since $U\cap V=\emptyset$, we deduce
\begin{align*}
2&=m'_Um_U(x)+m'_Vm_V(x) = m'\chi_Um(x)+m'\chi_Vm(x) \\ &= m'(\chi_U+\chi_V)m(x) = m'\chi_{U\cup V}m(x) = m'_{U\cup V}m_{U\cup V}(x)=1
\end{align*}
which is a contradiction.  Thus $y=z$.  The reverse implication is proved dually (by considering $mm'$).
\end{proof}

\begin{Cor}\label{c:isotropy.support}
Suppose that $R$ is indecomposable.  Let $\mathscr H$ be the interior of the isotropy bundle of the clopen subgroupoid $\mathscr G_1$ and let $m\in N_g$.  Then the containment \[\supp(m)\inv \supp(m)\cup \supp(m)\supp(m)\inv \subseteq\mathscr H\skel 1\] holds.
\end{Cor}
\begin{proof}
Obviously, $\supp(m)\inv \subseteq c(g\inv)$ and hence  $\supp(m)\inv\supp(m)$ and  $\supp(m)\supp(m)\inv$ are open subsets of $\mathscr G_1$.   If $\alpha,\beta\in \supp(m)$ with $\beta\inv\alpha$ defined, then $\ran(\alpha) =\ran(\beta)$ and so $\dom(\alpha)=\dom(\beta)$ by Proposition~\ref{p:isotropy}.  Thus $\beta\inv\alpha$ is an element of the isotropy group at $\dom(\alpha)$ and so we conclude that  $\supp(m)\inv\supp(m)\subseteq \mathscr H\skel 1$. The other inclusion is similar.
\end{proof}

Let us say that $\mathscr G$ satisfies the \emph{local bisection hypothesis} (relative to $R$ and the grading) if the support of each element of $N$ is a local bisection. For example, if $H$ is a group with the trivial grading and $R$ is indecomposable, then $N$ consists of the units of $RH$, together with $0$, and so $H$ satisfies the local bisection hypothesis if and only if $RH$ has no non-trivial units.

We continue to denote by $\mathscr H$ the interior of the isotropy bundle of $\mathscr G_1$.  Note that $D(\mathscr G)\subseteq R\mathscr H\subseteq R\mathscr G$ since $\mathscr H$ is an open subgroupoid containing all the objects.   We view $\mathscr H$ as trivially graded.   Let $N_{\mathscr H}$ denote the normalizer of $D(\mathscr G)$ in $R\mathscr H$.  The following proposition will allows us to reduce to the case of a group bundle when studying the local bisection property.

\begin{Prop}\label{p:reduced.to.iso}
Let $R$ be indecomposable and $\mathscr G$ a Hausdorff ample $G$-graded groupoid.  Let $\mathscr H$ be the interior of the isotropy bundle of $\mathscr G_1$.
\begin{enumerate}
\item $N_{\mathscr H} = N\cap R\mathscr H$.
\item $\mathscr G$ satisfies the local bisection hypothesis if and only if $\mathscr H$ does.
\end{enumerate}
\end{Prop}
\begin{proof}
Clearly, $N_{\mathscr H}\subseteq N\cap R\mathscr H$.  Suppose that $m\in N\cap R\mathscr H$; it suffices to show that $m'$ is supported on $\mathscr H\skel 1$.  Note that $m\in N_1$ and so $m'\in N_1$.    Let $\alpha\colon x\to y$ be in the support of $m'$.  Then $mm'(x)=1$ by Proposition~\ref{p:the.idems}.  Thus
\begin{equation}\label{eq:easy}
1=\sum_{\dom(\beta)=x} m(\beta\inv)m'(\beta).
\end{equation}
  Now if $\dom(\beta)=x$ and $\beta\in \supp(m')$, then $\ran(\beta)=y$ by Proposition~\ref{p:isotropy}.  If $x\neq y$, then $m(\beta\inv)=0$ for all $\beta\in \dom\inv(x)$, contradicting \eqref{eq:easy}.  We conclude that $x=y$ and so $\supp(m')\subseteq \mathscr H\skel 1$, establishing the first item.

 If $\mathscr G$ satisfies the local bisection hypothesis, then so does $\mathscr H$ as $N_{\mathscr H}\subseteq N$.  Suppose that $N_{\mathscr H}$ satisfies the local bisection hypothesis and let $m\in N_g$.  Suppose that $\alpha,\beta\in \supp(m)$ with $\dom(\alpha)=\dom(\beta)=x$.  Since $m$ is locally constant, there exists a compact local bisection $U$ with $\alpha\in U\subseteq \supp(m)$.  Then $U\inv \supp(m)\subseteq \mathscr H\skel 1$ by Corollary~\ref{c:isotropy.support} and so $n=\chi_{U\inv}m\in N\cap R\mathscr H= N_{\mathscr H}$.  Thus $\supp(n)$ is a local bisection.  On the other hand,
 \begin{align*}
 n(x) &= \sum_{\dom(\gamma)=x} \chi_{U\inv}(\gamma\inv)m(\gamma) = \chi_{U\inv}(\alpha\inv)m(\alpha)=m(\alpha)\neq 0\\ n(\alpha\inv\beta) &= \sum_{\dom(\gamma)=x}\chi_{U\inv}(\alpha\inv\beta\gamma\inv)m(\gamma)=\chi_{U\inv}(\alpha\inv)m(\beta)=m(\beta)\neq 0.
 \end{align*}
   Thus $\alpha=\beta$ as $\supp(n)$ is a local bisection.  Similarly, we have that $\ran|_{\supp(m)}$ is injective.
\end{proof}

As a consequence, if $\mathscr G_1$ is effective, then $\mathscr G$ satisfies the local bisection hypothesis.
\begin{Cor}\label{c:effective.case}
Suppose that $R$ is indecomposable, $\mathscr G$ is $G$-graded and Hausdorff and  $\mathscr G_1$ is effective. Then $\mathscr G$ satisfies the local bisection hypothesis.
\end{Cor}
\begin{proof}
The support of any locally constant function on $\mathscr G\skel 0$ is a local bisection and so the corollary is immediate from
 Proposition~\ref{p:reduced.to.iso} as $\mathscr G_1$ is effective.
\end{proof}

\subsection{Achieving the local bisection hypothesis}
In this subsection we provide some assumptions on a groupoid that imply the local bisection hypothesis. We continue to fix an ample Hausdorff $G$-graded groupoid $\mathscr G$ and to denote the interior of the isotropy bundle of $\mathscr G_1$ by $\mathscr H$. The isotropy group of $\mathscr H$ at an object $x$ will be denoted by $H_x$.

Proposition~\ref{p:reduced.to.iso} suggests that we should impose conditions on the group rings of the isotropy groups of $\mathscr H$.  Since a group $H$ satisfies the local bisection hypothesis if and only if $RH$ has only trivial units, we shall need to impose this on a dense set of objects of $\mathscr H$.  Recall that $N_{\mathscr H}$ is the normalizer of $D(\mathscr G)$ in $R\mathscr H$.

\begin{Prop}\label{p:good.isotropy}
Let $R$ be indecomposable and $m\in N_{\mathscr H}$.  Let $x\in \mathscr G\skel 0$ be such that $RH_x$ has no non-trivial units.  Then $|H_x\cap \supp(m)|\leq 1$.
\end{Prop}
\begin{proof}
Let $U=\dom(\supp(m))=\ran(\supp(m))$ (since $\dom=\ran$  for $\mathscr H$) and suppose that $x\in U$. Note that $\{x\}$ is a closed invariant subspace of $\mathscr H\skel 0$ and hence restriction to $H_x$ yields an $R$-algebra homomorphism $\rho_x\colon R\mathscr H\to RH_x$.   By Proposition~\ref{p:the.idems}, we have that  $m'm=\chi_U=mm'$ and so $\rho_x(m')\rho_x(m)=1=\rho_x(m)\rho_x(m')$.  Thus $\rho_x(m)$ is a unit and hence $\rho_x(m)$ has singleton support as $RH_x$ has no non-trivial units.  It follows that $|H_x\cap \supp(m)|=1$.
\end{proof}

We may now conclude that if $RH_x$ has only trivial units for a dense set of objects, then $\mathscr G$ satisfies the local bisection hypothesis.

\begin{Thm}\label{t:local.bis.from.iso}
Let $R$ be an indecomposable commutative ring with unit and let $\mathscr G$ be a Hausdorff ample $G$-graded groupoid.  Suppose that the set of objects $x\in \mathscr G\skel 0$ for which $RH_x$ has no non-trivial units is dense, where  $H_x$ denotes the isotropy group at $x$ in the interior of the isotropy bundle of $\mathscr G_1$.   Then $\mathscr G$ satisfies the local bisection hypothesis.
\end{Thm}
\begin{proof}
By Proposition~\ref{p:reduced.to.iso} we may assume without loss of generality that the grading is trivial and $\mathscr G$ is its own isotropy bundle.  Suppose that $m\in N$ and that $\alpha,\beta\in \supp(m)$ with $\dom(\alpha)=\dom(\beta)$ but $\alpha\neq \beta$.  Since $\mathscr G$ is Hausdorff and $m$ is locally constant, we can find disjoint compact local bisections $U,V$ with $\alpha\in U$, $\beta\in V$ and $U,V\subseteq \supp(m)$.  By assumption  $\dom(U)\cap \dom(V)$ is non-empty.  Hence there is $x\in \dom(U)\cap \dom (V)$ with $RH_x$ having no non-trivial units. By assumption there exists $\alpha'\in U$ and $\beta'\in V$ with $\dom(\alpha')=x=\dom(\beta')$, i.e., with $\alpha',\beta'\in H_x$.  But then $\alpha'=\beta'$ by Proposition~\ref{p:good.isotropy}.  This contradicts that $U$ and $V$ are disjoint.  We conclude that $\dom|_{\supp(m)}$ is injective. As $\dom=\ran$, this completes the proof.
\end{proof}

We end this section with some necessary conditions for the local bisection hypothesis to hold and give a class of groupoids where the sufficient condition from Theorem~\ref{t:local.bis.from.iso} is necessary.

\begin{Prop}\label{p:isolated}
Let $\mathscr G$ be a Hausdorff ample $G$-graded groupoid satisfying the local bisection hypothesis and $x\in \mathscr G\skel 0$ an isolated point. Let $G$ be the isotropy group at $x$ in $\mathscr G_1$.  Then $RG$ has no non-trivial units.
\end{Prop}
\begin{proof}
Since $x$ is isolated, $G$ is a clopen subgroupoid of $\mathscr G\skel 1_1$.  Let $f\in RG$ be a unit with inverse $g$.  Then $f\chi_Ug=\chi_U\chi_{\{x\}} = g\chi_Uf$ for $U\subseteq \mathscr G\skel 0$ compact open and so $f,g\in N_1$. It follows that $f$ has singleton support since $\mathscr G$ satisfies the local bisection hypothesis and so $f$ is a trivial unit.
\end{proof}

\begin{Cor}\label{c:necess.suff}
Let $R$ be an indecomposable commutative ring with unit and $\mathscr G$ a Hausdorff ample $G$-graded groupoid such that $\mathscr G\skel 0$ has a dense subset $X$ of isolated points.  Then $\mathscr G$ satisfies the local bisection hypothesis if and only if, for each $x\in X$, the group algebra over $R$ of the isotropy group of $x$ in $\mathscr G_1$ has no non-trivial units.
\end{Cor}
\begin{proof}
Necessity follows from Proposition~\ref{p:isolated} and sufficiency follows from Theorem~\ref{t:local.bis.from.iso}.
\end{proof}

\begin{Rmk}
If $\mathscr G_1$ is not effective, then in order for $\mathscr G$ to satisfy the local bisection hypothesis, $R$ must be reduced.  For if $0\neq n\in R$ is nilpotent and $U$ is a compact local bisection contained in the interior of the isotropy bundle of $\mathscr G_1$, but not inside $\mathscr G\skel 0$, then $\chi_{\dom(U)}-n\chi_U$ belongs to the normalizer of $D(\mathscr G)$ but its support is not a local bisection.  Indeed, if $n^k=0$, then \[(\chi_{\dom(U)}-n\chi_U)\chi_V(\chi_{\dom(U)}+\sum_{j=1}^{k-1}n^j\chi_U^j) = \chi_{V\cap \dom(U)}\] for any compact open $V\subseteq \mathscr G\skel 0$.

We also remark that the proof of Proposition~\ref{p:Neher} can be adapted to show that if $\mathscr G$ satisfies the local bisection hypothesis for every field $k$, then it satisfies it for every indecomposable reduced commutative ring with unit $R$.
\end{Rmk}

\section{Reconstructing the groupoid}\label{s:recons}
Let us assume now that $R$ is an indecomposable ring and $\mathscr G$ is a $G$-graded ample Hausdorff groupoid.  We continue to denote the normalizer of the diagonal subalgebra by $N$.
We wish to recover $\mathscr G$ (or, equivalently, $\Bis_c^h(\mathscr G)$) from the pair $(R\mathscr G,D(\mathscr G))$ if we impose the local bisection hypothesis.  Although $N$ contains $\Bis_c^h(\mathscr G)$ they are never equal because we can, for example, multiply by elements of $R^\times$. For instance, if $G$ is a group (with trivial grading), the diagonal subalgebra of $RG$ is $R$ and the normalizer of the diagonal subalgebra is the unit group $(RG)^\times$, together with $0$.  The local bisection hypothesis becomes the assumption that $(RG)^\times = R^\times G$ and so to recover $G$ we must factor out by $R^\times$.  Since we want to be able to recover our groupoid from diagonal-preserving ring isomorphisms, rather than $R$-algebra isomorphisms, we must think of $R^\times$ as the (non-zero) diagonal elements of $N$.  Thus, in the general case, we need to factor out by the diagonal elements of $N$ (and impose the local bisection hypothesis).   To do this, we shall need to recall how to factor an inverse semigroup by a suitable normal inverse subsemigroup.

An inverse subsemigroup $K$ of an inverse semigroup $S$ is \emph{normal} if $K$ is full (so $E(K)=E(S)$) and $sKs^*\subseteq K$ for all $s\in S$.  If $\p\colon S\to T$ is an inverse semigroup homomorphism, then its \emph{kernel} $\ker \p=\pinv(E(T))$ is a normal subsemigroup.  If $\sim$ is a congruence on $S$, then the kernel of $\sim$ will be the kernel of the quotient $S\to S/{\sim}$.  A congruence $\sim$ is \emph{idempotent separating} if $S\to S/{\sim}$ is injective on $E(S)$.  It is \emph{idempotent pure} if its kernel is $E(S)$.  It is well known and easy to check that $\sim$ is the equality relation if and only if it is both idempotent separating and idempotent pure; see~\cite[Chpt.~5]{Lawson}.    An idempotent separating congruence is completely determined by its kernel.  Moreover, a normal subsemigroup $K$ is the kernel of an idempotent separating congruence if and only if $a^*a=aa^*$ for all $a\in K$~\cite[Lemma~5.1]{Lawson}; the corresponding congruence is given by $s\sim t$ if $s^*s=t^*t$ and $st^*\in K$ (for a group, this reduces to the usual notion of factoring out a normal subgroup).  See~\cite[Chpt.~5]{Lawson} for details.

Consider now $K=N\cap D(\mathscr G)$.  
It is an inverse subsemigroup of $N$ by Corollary~\ref{c:diag.inv.sub}, which, moreover, is full because $E(N)\subseteq D(\mathscr G)$ by Remark~\ref{r:where.the.idems.are}.  It is normal by definition of $N$.  Furthermore, since $D(\mathscr G)$ is commutative, $a'a=aa'$ for all $a\in K$.  Thus $K$ is the kernel of an idempotent separating congruence $\sim$ given by $m\sim n$ if $m'm=n'n$ and $mn'\in D(\mathscr G)$.  Note that $0$ is in a congruence class of its own and that if $m,n\in N\setminus \{0\}$ with $m\sim n$, then $mn'\in D(\mathscr G)\subseteq R\mathscr G_1$ implies that if $m\in N_g$, then $n\in N_g$.  Thus $\sim$ respects the grading.  In summary, we have the following proposition.

\begin{Prop}\label{p:is.equiv}
The relation $\sim$ is an idempotent separating congruence on $N$ with kernel $N\cap D(\mathscr G)$ that respects the grading.
\end{Prop}
%
%

\begin{Rmk}\label{r:determineNmodsim}
Observe that $N/{\sim}$ is determined up to isomorphism of $G$-graded inverse semigroups by $R\mathscr G$, taken up to graded diagonal-preserving isomorphism of rings by construction.
\end{Rmk}

We shall write $[m]$ for the class of $m$ in $N/{\sim}$.

\begin{Prop}\label{p:doesnot.ident}
Let $R$ be indecomposable.  The $G$-graded homomorphism  $\psi\colon \Bis_c^h(\mathscr G)\to N/{\sim}$ given by $\psi(U)=[\chi_U]$ is injective.  It is an isomorphism if and only if $\mathscr G$ satisfies the local bisection hypothesis.
\end{Prop}
\begin{proof}
The congruence $\sim$ is idempotent separating.  Its restriction to $\Bis_c^h(\mathscr G)$ is idempotent pure because if $\chi_U\in K=N\cap D(\mathscr G)$, then $U\subseteq \mathscr G_0$. It follows that $\psi$ is idempotent separating and idempotent pure, hence injective.

Suppose now that $\mathscr G$ satisfies the local bisection hypothesis and let $m\in N_g$.  Then the support $V$ of $m$ is a compact local bisection contained in $c\inv(g)$.  If $m'$ is as per \eqref{eq:normalizer}, then $m'm=\chi_{\dom(\supp(m))}=\chi_{V\inv V}=\chi_V'\chi_V$ by Proposition~\ref{p:the.idems}.  Also $m\chi_V'=m\chi_{V\inv}$ has support contained in $\supp(m)V\inv = VV\inv\subseteq \mathscr G\skel 0$ and so $m\chi_V'\in D(\mathscr G)$.  Thus $m\sim \chi_V$ and so $[m]=\psi(\chi_V)$.

Conversely, if $\psi$ is surjective and $m\in N_g$, then there exists $V\in \Bis_c^h(\mathscr G)_g$ with $m\sim \chi_V$. Let $m'$ be as per \eqref{eq:normalizer}.   Then $m'm= \chi_{V\inv}\chi_V$ and $m\chi_{V\inv}\in D(\mathscr G)$.  So $m=mm'm=m\chi_{V\inv}\chi_V=f\chi_V$ with $f\in D(\mathscr G)$.  Thus $\supp(m)\subseteq V$ and hence $\supp(m)$ is a local bisection.	
\end{proof}

\begin{Rmk}\label{r:nature.of.proj}
The previous argument shows that if $\mathscr G$ satisfies the local bisection hypothesis, then $m\sim \chi_{\supp(m)}$ for $m\in N$.
\end{Rmk}

We now aim to show that if $\mathscr G$ and $\mathscr G'$ are ample $G$-graded Hausdorff groupoids such that $\mathscr G$ satisfies the local bisection hypothesis and there is a diagonal-preserving ring isomorphism $\Phi\colon R\mathscr G\to R\mathscr G'$, then $\mathscr G'$ also satisfies the local bisection hypothesis.   The argument is loosely inspired by what happens in the special case of groups~\cite[Chpt.~14, Thm~3.1]{PassmanBook}.

\begin{Thm}\label{t:robust}
Suppose that $R$ is an indecomposable commutative ring with unit and that $\mathscr G_1$, $\mathscr G_2$ are ample $G$-graded Hausdorff groupoids. If $\mathscr G_2$ satisfies the local bisection hypothesis and there is a graded diagonal-preserving ring isomorphism $\Phi\colon R\mathscr G_1\to R\mathscr G_2$, then $\mathscr G_1$ satisfies the local bisection hypothesis.	
\end{Thm}
\begin{proof}
Let us denote the identity of $G$ by $e$.  First note that, as $(R\mathscr G_i)_e=R(\mathscr G_i)_e$, for $i=1,2$, it follows that $\Phi$ restricts to a diagonal-preserving ring isomorphism $R(\mathscr G_1)_e\to R(\mathscr G_2)_e$.   But then, by Corollary~\ref{c:isotropy.pres}, if $\mathscr H_i$ is the interior of the isotropy bundle of $(\mathscr G_i)_e$ for $i=1,2$, then $\Phi$ restricts to a diagonal-preserving ring isomorphism $R\mathscr H_1\to R\mathscr H_2$.
Thus we can assume without loss of generality, by Proposition~\ref{p:reduced.to.iso}, that $\mathscr G_1,\mathscr G_2$ are group bundles and the grading is trivial.  Then $R\mathscr G_i$ is a $D(\mathscr G_i)$-algebra, for $i=1,2$, by Proposition~\ref{p:z.of.diag}.

Let $N_i$  be the normalizer of the diagonal subalgebra in $\mathscr G_i$ and $\psi_i\colon N_i\to N_i/{\sim}$ the canonical projection, for $i=1,2$.  By Proposition~\ref{p:doesnot.ident}, it suffices to prove that $\psi_1(\Bis_c(\mathscr G_1))=N_1/{\sim}$.  Observe that $\Phi$, being diagonal preserving, induces an isomorphism $\p\colon N_1/{\sim}\to N_2/{\sim}$ such that
\[
\begin{tikzcd}
N_1\ar{r}{\Phi}\ar{d}[swap]{\psi_1} & N_2\ar{d}{\psi_2}\\ N_1/{\sim}\ar{r}[swap]{\p} & N_2/{\sim}
\end{tikzcd}
\]
commutes.  Let $T=\Phi(\Bis_c(\mathscr G_1))$.  It then suffices to show that $\psi_2(T)=N_2/{\sim}$.  Note that $\psi_2(T)$ is full in $N_2/{\sim}$ since $\psi_1(\Bis_c(\mathscr G_1))$ is full in $N_1/{\sim}$.

First observe that $\Bis_c(\mathscr G_1)$ spans $R\mathscr G_1$ over $D(\mathscr G_1)$.  Indeed, if $f\in R\mathscr G_1$, we can write $f=\sum_{i=1}^n c_i\chi_{U_i}$ with the $c_i\in R$ and $U_i\in \Bis_c(\mathscr G_1)$.  Then $g_i=c_i\chi_{U_iU_i\inv}\in D(\mathscr G_1)$ and $f=\sum_{i=1}^n g_i\chi_{U_i}$.
It follows that $T$ spans $R\mathscr G_2$ over $D(\mathscr G_2)$.

By Remark~\ref{r:nature.of.proj}, we have that $\psi_2(m)=\psi_2(\chi_{\supp(m)})$ for $m\in N_2$. So it suffices to show that if $V\in \Bis_c(\mathscr G_2)$, then $\psi_2(\chi_V)\in \psi_2(T)$.
If $\gamma\in V$, then since $\chi_V$ is in the $D(\mathscr G_2)$-span of $T$, there is $m\in T$ with $m(\gamma)\neq 0$. Now $U=\supp(m)\in \Bis_c(\mathscr G_2)$ by the local bisection hypothesis.  Hence $\gamma\in U\cap V$ and $U\cap V\in \Bis_c(\mathscr G_2)$ since $\mathscr G_2$ is Hausdorff.  Notice that $\psi_2(\chi_{U\cap V})\leq \psi_2(\chi_U)=\psi_2(m)$ and hence, since $\psi_2(T)$ is full and thus an order ideal, $\psi_2(\chi_{U\cap V})\in \psi_2(T)$.  Thus we can find $n\in T$ with $\gamma\in \supp(n)\subseteq V$ for each $\gamma\in V$.

By compactness of $V$, we can then find $m_1,\ldots, m_k\in T$  such that $V=\bigcup_{i=1}^k U_i$ with $U_i=\supp(m_i)$.  Since $\psi_2\colon \Bis_c(\mathscr G_2)\to N_2/{\sim}$ is an isomorphism by Proposition~\ref{p:doesnot.ident} and  $\Bis_c(\mathscr G_2)$ is distributive,  we conclude that $N_2/{\sim}$ is distributive and  $\psi_2(\chi_V) = \bigvee_{i=1}^k \psi_2(\chi_{U_i}) = \bigvee_{i=1}^k\psi_2(m_i)$.  But $\psi_2(T)$ is full in $N_2/{\sim}$ and distributive (being isomorphic to $\Bis_c(\mathscr G_1)$), and hence is closed under joins, as discussed in Subsection~\ref{ss:distributive}.  We conclude that $\psi_2(\chi_V)\in \psi_2(T)$.  This concludes the proof that $\mathscr G_1$ satisfies the local bisection hypothesis.
\end{proof}

We can now state the main theorem of this section.

\begin{Thm}\label{t:construct.from.bisection}
Let $R$ be an indecomposable commutative ring with unit and let $\mathscr G$ and $\mathscr G'$ be $G$-graded Hausdorff ample groupoids such that $\mathscr G$ satisfies the local bisection hypothesis. Then the following are equivalent.
\begin{enumerate}
\item There is a graded isomorphism $\p\colon \mathscr G\to \mathscr G'$.
\item There is a diagonal-preserving graded isomorphism $\Phi\colon R\mathscr G\to R\mathscr G'$ of $R$-algebras.
\item There is a diagonal-preserving graded isomorphism $\Phi\colon R\mathscr G\to R\mathscr G'$ of rings.
\end{enumerate}
\end{Thm}
\begin{proof}
Trivially, (1) implies (2) implies (3).  Suppose that (3) holds.  Then $\mathscr G'$ satisfies the local bisection hypothesis by Theorem~\ref{t:robust}.  Let $N, N'$ be the normalizers of the diagonal subalgebra in $R\mathscr G, R\mathscr G'$, respectively.  As $\Phi$ is a diagonal-preserving  graded ring isomorphism, it follows that it induces a graded isomorphism of inverse semigroups $N/{\sim}\to N'/{\sim}$ by Remark~\ref{r:determineNmodsim}. Thus $\Bis_c^h(\mathscr G)\cong \Bis_c^h(\mathscr G')$ as $G$-graded inverse semigroups by Proposition~\ref{p:doesnot.ident}.  We conclude that $\mathscr G\cong \mathscr G'$ as $G$-graded groupoids by Theorem~\ref{t:graded.reconstruct.inv.sgp}.
\end{proof}

As a consequence of Theorem~\ref{t:local.bis.from.iso} and Theorem~\ref{t:construct.from.bisection} we may now state the main result of this paper.

\begin{Thm}\label{t:main}
Let $R$ be an indecomposable commutative ring with unit and let $\mathscr G$, $\mathscr G'$ be Hausdorff ample $G$-graded groupoids. Suppose that $\mathscr G$ has a dense set of objects $x$ such that the group algebra over $R$ of the isotropy group at $x$ of the interior of the isotropy bundle of the homogeneous component of $1$ has no non-trivial units.  Then the following are equivalent.
\begin{enumerate}
\item There is a graded isomorphism $\p\colon \mathscr G\to \mathscr G'$.
\item There is a diagonal-preserving graded isomorphism $\Phi\colon R\mathscr G\to R\mathscr G'$ of $R$-algebras.
\item There is a diagonal-preserving graded isomorphism $\Phi\colon R\mathscr G\to R\mathscr G'$ of rings.
\end{enumerate}
\end{Thm}

The hypotheses of Theorem~\ref{t:main} are met by a groupoid $\mathscr G$ if $\mathscr G_1$ is effective (and, in particular, if $\mathscr G_1$ is topologically principal), so we recover the main results of~\cite{reconstruct}. In~\cite{CarlsenSteinberg}, Carlsen and Rout prove a version of Theorem~\ref{t:main} under the stronger hypotheses that $R$ is an integral domain and that there is a dense set of objects such $x$ such that the group algebra of the isotropy group at $x$ in $\mathscr G_1$ has no zero divisors and no non-trivial units.  So our results imply those of~\cite{CarlsenSteinberg}.  Note that if $\mathscr G_1$ is not effective, then $R$ must be reduced for the hypotheses of the theorem to apply.    Of course, Theorem~\ref{t:main} applies to the ungraded setting by taking $G$ to be trivial.

For effective groupoids, using that the diagonal subalgebra is a maximal commutative subring, we can, in fact, formulate a slightly stronger result.
Since topologically principal groupoids are effective~\cite{operatorsimple1}, the following theorem recovers, and extends, the main result of~\cite{reconstruct}.

\begin{Cor}\label{c:effective.case.gpd.rec}
Let $R$ be an indecomposable commutative ring with unit and let $\mathscr G$ and $\mathscr G'$ be $G$-graded Hausdorff ample groupoids.  Suppose that $\mathscr G_1$ is effective.  Then the following are equivalent.
\begin{enumerate}
\item There is a graded isomorphism $\p\colon \mathscr G\to \mathscr G'$.
\item There is a diagonal-preserving graded isomorphism $\Phi\colon R\mathscr G\to R\mathscr G'$ of $R$-algebras.
\item There is a diagonal-preserving graded isomorphism $\Phi\colon R\mathscr G\to R\mathscr G'$ of rings.
\item There is a graded $R$-algebra isomorphism $\Phi\colon R\mathscr G\to R\mathscr G'$ with $\Phi(D(\mathscr G))\subseteq D(\mathscr G')$.
\item There is a graded ring isomorphism $\Phi\colon R\mathscr G\to R\mathscr G'$ with $\Phi(D(\mathscr G))\subseteq D(\mathscr G')$.
\end{enumerate}
\end{Cor}
\begin{proof}
The equivalence of (1)--(3) follows from Theorem~\ref{t:construct.from.bisection} and Corollary~\ref{c:effective.case}. Clearly the fourth item implies the fifth.  Suppose the fifth item holds.  By Corollary~\ref{c:effective.diag.det},  $D(\mathscr G)$ is a maximal commutative subring of $R\mathscr G_1$ and hence $\Phi(D(\mathscr G))$ is a maximal commutative subring of $R\mathscr G'_1$.  But $\Phi(D(\mathscr G))\subseteq D(\mathscr G')\subseteq R\mathscr G'_1$ and $D(\mathscr G')$ is commutative.  Thus $\Phi(D(\mathscr G))=D(\mathscr G')$ and so the third item holds.
\end{proof}

Notice that in~\cite{reconstruct} both groupoids are assumed to have a topologically principle homogeneous component of $1$, whereas  our result only requires that $\mathscr G_1$ is effective: no assumption is made on $\mathscr G'_1$.

In~\cite{operatorsimple1}, an example is given of an effective ample groupoid $\mathscr G$ such that each isotropy group is infinite cyclic.  If $R$ is any indecomposable ring that is not reduced, then  the group ring of each isotropy group of $\mathscr G$ has zero divisors and non-trivial units.  Nonetheless, Corollary~\ref{c:effective.case.gpd.rec} applies to $\mathscr G$.

\section{An application to Leavitt path algebras}
We end this paper with an application to Leavitt path algebras.
Let $E$ be a (directed) graph with vertex set $E\skel 0$ and edge set $E\skel 1$.  We use $\sour(e)$ for the source of an edge $e$ and $\ran(e)$ for the range, or target, of an edge.   A vertex $v$ is a \emph{sink} if $\sour\inv(v)=\emptyset$ and an \emph{infinite emitter} if $|\sour\inv(v)|=\infty$.  The length of a finite (directed) path $\alpha$ is denoted $|\alpha|$.

The \emph{Leavitt path algebra}~\cite{LeavittBook,Leavitt1,LeavittPardo,Abramssurvey} $L_R(E)$ of $E$ with coefficients in $R$ is the $R$-algebra generated by a set $\{v\in E\skel 0\}$ of pairwise orthogonal idempotents and a set of variables $\{e,e^*\mid e\in E\skel 1\}$ satisfying the relations:
\begin{enumerate}
\item $\sour(e)e=e=e\ran(e)$ for all $e\in E\skel 1$;
\item $\ran(e)e^*= e^*=e^*\sour(e)$ for all $e\in E\skel 1$;
\item $e^*e'=\delta_{e,e'}\ran(e)$ for all $e,e'\in E\skel 1$;
\item $v=\sum_{e\in \sour^{-1}(v)} ee^*$ whenever $v$ is not a sink and not an infinite emitter.
\end{enumerate}

It is well known that $L_R(E)\cong R\mathscr G_E$  for the path groupoid $\mathscr G_E$ defined as follows, cf.~\cite{GroupoidMorita}.  Let $\partial E$ consist of all one-sided infinite paths in $E$ as well as all finite paths $\alpha$ ending in a vertex $v$ that is either a sink or an infinite emitter.  If $\alpha$ is a finite path in $E$ (possibly empty), put $Z(\alpha)=\{\alpha\beta\in \partial E\}$ (if $\alpha$ is the empty path $\varepsilon_v$ at $v$, this should be interpreted as those elements of $\partial E$ with initial vertex $v$).  
Then a basic open neighborhood  of $\partial E$ is of the form $Z(\alpha)\setminus (Z(\alpha e_1)\cup\cdots \cup Z(\alpha e_n))$ with $e_i\in E\skel 1$, for $i=1,\ldots n$ (and possibly $n=0$).  These neighborhoods are compact open.

The graph groupoid $\mathscr G_E$ is the given by:
\begin{itemize}
\item $\mathscr G_E\skel 0= \partial E$;
\item $\mathscr G_E\skel 1 =\{ (\alpha\gamma,|\alpha|-|\beta|,\beta\gamma)\in \partial E\times \mathbb Z\times \partial E\}\mid |\alpha|,|\beta|<\infty\}$.
\end{itemize}
One has $\dom(\eta,k,\gamma)=\gamma$, $\ran(\eta,k,\gamma)=\eta$ and $(\eta,k,\gamma)(\gamma,m,\xi) = (\eta, k+m,\xi)$.  The inverse of $(\eta,k,\gamma)$ is $(\gamma,-k,\eta)$.  The projection $c\colon \mathscr G_E\to \mathbb Z$ given by $(\eta,k,\gamma)\mapsto k$ is a continuous cocycle.

A basis of compact open subsets for the topology on $\mathscr G_E\skel 1$ can be described as follows. Let $\alpha,\beta$ be finite paths ending at the same vertex and let $U\subseteq Z(\alpha)$, $V\subseteq Z(\beta)$ be compact open with $\alpha\gamma\in U$ if and only if $\beta\gamma\in V$.  Then the set \[(U,\alpha,\beta,V)=\{\alpha\gamma,|\alpha|-|\beta|,\beta\gamma)\mid \alpha\gamma\in U,\beta\gamma\in V\}\] is a basic compact open set of $\mathscr G_E\skel 1$. Of particular importance are the compact open sets $Z(\alpha,\beta) = (Z(\alpha),\alpha,\beta,Z(\beta))= \{(\alpha\gamma,|\alpha|-|\beta|,\beta\gamma)\in \mathscr G\skel 1\}$ where $\alpha,\beta$ are finite paths ending at the same vertex.
It is well known that $\mathscr G_E$ is Hausdorff.

There is an isomorphism $L_R(E)\to R\mathscr G_E$ sending $v\in E\skel 0$ to the characteristic function of $Z(\varepsilon_v,\varepsilon_v)$ and, for $e\in E\skel 1$, sending $e$ to the characteristic function of $Z(e,\varepsilon_{\ran(e)})$ and $e^*$ to the characteristic function of $Z(\varepsilon_{\ran(e)},e)$, cf.~\cite{operatorguys1,groupoidprimitive,CRS17,Strongeffective} or~\cite[Example~3.2]{GroupoidMorita}.

By a \emph{cycle} in a directed graph $E$, we mean a simple, directed,  closed circuit.  A cycle is said to have an \emph{exit} if some vertex on the cycle has out-degree at least two.  It is well known that the isotropy group at an element $\gamma\in \partial E$ is trivial unless $\gamma$ is \emph{eventually periodic}, that is, $\gamma=\rho\alpha\alpha\cdots$ with $\alpha$ a cycle, in which case the isotropy group is infinite cyclic, cf.~\cite{gpdchain}.
Proposition~\ref{p:Neher} therefore implies that the group algebra of each isotropy group of $\mathscr G_E$ over an indecomposable reduced ring has no non-trivial units.  If $E$ satisfies condition (L), that every cycle has an exit, then $\mathscr G_E$ is well known to be topologically principal. Thus Theorem~\ref{t:main} has the following specialization to Leavitt path algebras.  It extends the results of~\cite{BCvH2017,CarlsenSteinberg} in that we weaken the integral domain hypothesis and we require only one of the groupoids to be a path groupoid.
Applications of the previous versions of this result can be found in~\cite{CRS17}.

\begin{Thm}\label{t:main.Leavitt}
Let $R$ be an indecomposable reduced commutative ring with unit and let $E$ be a graph with associated path groupoid $\mathscr G_E$. Let $\mathscr G$ be any ample Hausdorff ($\mathbb Z$-graded) groupoid.   Then the following are equivalent.
\begin{enumerate}
\item There is a (graded) isomorphism $\p\colon \mathscr G_E\to \mathscr G$.
\item There is a diagonal-preserving (graded) isomorphism $\Phi\colon L_R(E)\to R\mathscr G$ of $R$-algebras.
\item There is a diagonal-preserving (graded) isomorphism $\Phi\colon L_R(E)\to R\mathscr G$ of rings.
\end{enumerate}
The same result holds over any indecomposable commutative ring with unit $R$ if $E$ is assumed to satisfy condition (L).
\end{Thm}

An important special case of Theorem~\ref{t:main.Leavitt} is when $\mathscr G$ is also a path groupoid.

\begin{appendix}
\section{The non-Hausdorff case}
In this appendix we extend Theorem~\ref{t:construct.from.bisection} to the case where $\mathscr G'$ is not assumed Hausdorff. Let $R$ be a commutative ring with unit.  If $\mathscr G$ is an ample groupoid with $\mathscr G\skel 1$ not necessarily Hausdorff (but $\mathscr G\skel 0$ is assumed Hausdorff), we can still define the groupoid algebra $R\mathscr G$.  In this case, we let $\Bis_c(\mathscr G)$ be the set of compact Hausdorff local bisections.  It is still an inverse semigroup and a basis for the topology on $\mathscr G\skel 1$.  Note that elements of $\Bis_c(\mathscr G)$ are open but not necessarily closed.  Define $R\mathscr G$ to be the $R$-span of the characteristic functions $\chi_U$ with $U\in \Bis_c(\mathscr G)$ equipped with the convolution product.  This is still an $R$-algebra and coincides with our previous definition when $\mathscr G\skel 1$ is Hausdorff.  See~\cite{mygroupoidalgebra} for details.  The diagonal subalgebra $D(\mathscr G)$ is the span of the characteristic functions of compact open subsets of $\mathscr G\skel 0$, which is still the commutative algebra of locally constant $R$-valued functions on $\mathscr G\skel 0$ with pointwise product.

Our goal is to show that if $\mathscr G$ is a Hausdorff ample groupoid satisfying the local bisection hypothesis (relative to $R$) and $\mathscr  G'$ is a (not necessarily Hausdorff) ample groupoid, then the existence of a diagonal-preserving isomorphism  between $R\mathscr G$ and $R\mathscr G'$ implies that $\mathscr G$ and $\mathscr G'$ are isomorphic. 

The following proposition is well known.

\begin{Prop}\label{p:isHauss}
An ample groupoid $\mathscr G$ is Hausdorff if and only if $\mathscr G\skel 0$ is a closed subspace of $\mathscr G\skel 1$.
\end{Prop}
\begin{proof}
If $\mathscr G\skel 1$ is Hausdorff, then since $\mathscr G\skel 0$ is the equalizer of the domain map and the identity map on $\mathscr G\skel 1$, it is clearly closed.  Suppose that $\mathscr G\skel 0$ is closed.  Let $\alpha\neq\beta\in \mathscr G\skel 1$ and suppose first that $\dom(\alpha)\neq \dom(\beta)$. If $U,V$ are disjoint neighborhoods of $\dom(\alpha)$ and $\dom(\beta)$, respectively, in $\mathscr G\skel 0$, then $\dom\inv(U)$ and $\dom\inv(V)$ are disjoint neighborhoods of $\alpha,\beta$, respectively.  So next assume that $\dom(\alpha)=\dom(\beta)$, whence $\alpha\beta\inv\in \mathscr G\skel 1\setminus G\skel 0$.  Let $U$ be a local bisection containing $\beta$.  Then $(\mathscr G\skel 1\setminus G\skel 0)U$ and $U$ are disjoint neighborhoods of $\alpha,\beta$ respectively.  This completes the proof.
\end{proof}

Next we verify that a $G$-graded groupoid $\mathscr G$ is Hausdorff if and only if $\Bis^h_c(\mathscr G)$ has binary meets.

\begin{Prop}\label{p:meet.Haus}
Let $\mathscr G$ be an ample $G$-graded groupoid.  Then $\mathscr G\skel 1$ is Hausdorff if and only if $\Bis^h_c(\mathscr G)$ admits binary meets.
\end{Prop}
\begin{proof}
If $\mathscr G\skel 1$ is Hausdorff and $U,V\in \Bis_c(\mathscr G)$, then $U\cap V$ is closed in $U$ and hence compact.  Thus $U\cap V\in \Bis_c(\mathscr G)$ and so $\Bis_c(\mathscr G)$ admits binary meets.  As $\Bis^h_c(\mathscr G)$ is a full inverse subsemigroup, and hence an order ideal in $\Bis_c(\mathscr G)$, it also admits binary meets.  
  Suppose that $\Bis^h_c(\mathscr G)$ has binary meets.  We show that $\mathscr G\skel 0$ is closed in $\mathscr G\skel 1$.  Let $\gamma\in \mathscr G\skel 1\setminus \mathscr G\skel 0$.  Choose $U\in \Bis^h_c(\mathscr G)$ with $\gamma\in U$ and let $V$ be the meet in $\Bis_c^h(\mathscr G)$ of $U$ and $U\inv U$.  Then $V\subseteq U\inv U\subseteq \mathscr G\skel 0$ and $V\subseteq U$.  Since $U$ is Hausdorff and $V$ is compact, we deduce that $V$ is closed in $U$ and hence $U\setminus V$ is open.  Clearly, $\gamma\in U\setminus V$.  Suppose that $x\in U\cap \mathscr G\skel 0$.  Then $x\in U\cap U\inv U$ and there is a compact open set $W$ with $x\in W\subseteq U\cap U\inv U\subseteq \mathscr G\skel 0$.  Thus $W\in \Bis^h_c(\mathscr G)$ is a common lower bound of $U,U\inv U$ and so $W\subseteq V$.  It follows that $U\setminus V\subseteq \mathscr G\skel 1\setminus \mathscr G\skel 0$.  We conclude that $\mathscr G\skel 0$ is closed and so $\mathscr G$ is Hausdorff by Proposition~\ref{p:isHauss}.
\end{proof}

We observe that all the results of Subsection~\ref{ss:local.bis} up to, and including, Corollary~\ref{c:diag.inv.sub} do not use the Hausdorff assumption and hence are valid for non-Hausdorff groupoids.  Also the results of Section~\ref{s:recons} up to, and including, Proposition~\ref{p:doesnot.ident} do not use the Hausdorff assumption.  Thus we have that the normalizer $N$ of the diagonal subalgebra is an inverse semigroup in the non-Hausdorff case and that $N\cap D(\mathscr G)$ is a normal inverse subsemigroup such that $\Bis^h_c(\mathscr G)$ embeds as a full inverse subsemigroup of $N/{\sim}$ (where $\sim$ is the idempotent-separating congruence associated to $N\cap D(\mathscr G)$).  We then have the following extension of Theorem~\ref{t:construct.from.bisection}.

\begin{Thm}\label{t:construct.from.bisection2}
Let $R$ be an indecomposable commutative ring with unit and let $\mathscr G$ be a $G$-graded Hausdorff ample groupoid satisfying the local bisection hypothesis. Let $\mathscr G'$ be any ample groupoid (not necessarily Hausdorf).  Then the following are equivalent.
\begin{enumerate}
\item There is a graded isomorphism $\p\colon \mathscr G\to \mathscr G'$.
\item There is a diagonal-preserving graded isomorphism $\Phi\colon R\mathscr G\to R\mathscr G'$ of $R$-algebras.
\item There is a diagonal-preserving graded isomorphism $\Phi\colon R\mathscr G\to R\mathscr G'$ of rings.
\end{enumerate}
\end{Thm}
\begin{proof}
Trivially, (1) implies (2) implies (3).  Suppose that (3) holds.   Let $N, N'$ be the normalizers of the diagonal subalgebra in $R\mathscr G, R\mathscr G'$, respectively.  As $\Phi$ is a diagonal-preserving  graded ring isomorphism, it follows that it induces an somorphism of inverse semigroups $\Bis_c^h(\mathscr G)\cong N/{\sim}\to N'/{\sim}$ by Proposition~\ref{p:doesnot.ident}  and Remark~\ref{r:determineNmodsim}.   Thus $\Bis_c^h(\mathscr G')$ is isomorphic to a full inverse subsemigroup of the inverse semigroup $\Bis^h_c(\mathscr G)$.  But $\Bis^h_c(\mathscr G)$ admits binary meets by Proposition~\ref{p:meet.Haus} and hence so does $\Bis_c^h(\mathscr G')$, as it is an order ideal in an inverse semigroup with binary meets (and hence closed under binary meets).  Therefore, $\mathscr G'$ is Hausdorff by Proposition~\ref{p:meet.Haus}.  We can now apply Theorem~\ref{t:construct.from.bisection}.
\end{proof}

As a consequence, we obtain the following generalization of Theorem~\ref{t:main}.

\begin{Thm}\label{t:main2}
Let $R$ be an indecomposable commutative ring with unit,  $\mathscr G$ a Hausdorff ample $G$-graded groupoid and  $\mathscr G'$ any ample groupoid (not necessarily Hausdorff). Suppose that $\mathscr G$ has a dense set of objects $x$ such that the group algebra over $R$ of the isotropy group at $x$ of the interior of the isotropy bundle of the homogeneous component of $1$ has no non-trivial units.  Then the following are equivalent.
\begin{enumerate}
\item There is a graded isomorphism $\p\colon \mathscr G\to \mathscr G'$.
\item There is a diagonal-preserving graded isomorphism $\Phi\colon R\mathscr G\to R\mathscr G'$ of $R$-algebras.
\item There is a diagonal-preserving graded isomorphism $\Phi\colon R\mathscr G\to R\mathscr G'$ of rings.
\end{enumerate}
\end{Thm}

Similarly, the Hausdorff assumption can be removed on $\mathscr G$ in Theorem~\ref{t:main.Leavitt}.
\end{appendix}
\def\malce{\mathbin{\hbox{$\bigcirc$\rlap{\kern-7.75pt\raise0,50pt\hbox{${\tt
  m}$}}}}}\def\cprime{$'$} \def\cprime{$'$} \def\cprime{$'$} \def\cprime{$'$}
  \def\cprime{$'$} \def\cprime{$'$} \def\cprime{$'$} \def\cprime{$'$}
  \def\cprime{$'$} \def\cprime{$'$}

\end{document}